\documentclass{amsart}
\usepackage{mathrsfs,amssymb,setspace}

\theoremstyle{plain}
\newtheorem{theorem}{Theorem}[section]
\newtheorem{lemma}[theorem]{Lemma}

\theoremstyle{definition}
\newtheorem{notation}[theorem]{Notation}
\newtheorem{example}[theorem]{Example}
\newtheorem{definition}[theorem]{Definition}

\theoremstyle{remark}
\newtheorem{remark}[theorem]{Remark}

\input xy
\xyoption{all}

\newcommand{\lref}[1]{Lemma \ref{#1}}
\newcommand{\tref}[1]{Theorem \ref{#1}}

\newcommand{\cref}[1]{Corollary \ref{#1}}
\newcommand{\rref}[1]{Remark \ref{#1}}

\newcommand{\fref}[1]{Figure \ref{#1}}

\def\vt{\vartheta}

\newcommand{\sis}[2]{^{#1}\!\zeta_{#2}}

\begin{document}


\title[Radicals and Ideals of Affine Near-semirings over Brandt Semigroups]{Radicals and Ideals of Affine Near-semirings \\ over Brandt Semigroups}
\author[Jitender Kumar, K. V. Krishna]{Jitender Kumar and  K. V. Krishna}
\address{Department of Mathematics, Indian Institute of Technology Guwahati, Guwahati, India}
\email{\{jitender, kvk\}@iitg.ac.in}


\begin{abstract}
This work obtains all the right ideals, radicals, congruences and ideals of the affine near-semirings over Brandt semigroups.
\end{abstract}



\maketitle

\section{Introduction}

An algebraic structure $(\mathcal{N}, +, \cdot)$ with two binary operations $+$ and $\cdot$ is said to be a near-semiring if $(\mathcal{N}, +)$ and $(\mathcal{N}, \cdot)$ are semigroups and $\cdot$ is one-side, say left, distributive over $+$, i.e. $a \cdot (b + c) = a \cdot b + a \cdot c$, for all $a,b,c \in \mathcal{N}$. Typical examples of near-semirings are of the form $M(S)$, the set of all mappings on a semigroup $S$. Near-semirings are not only the natural generalization of semirings and near-rings, but also they have very prominent applications in computer science. To name a few: process algebras by Bergstra and Klop \cite{b.bergstra90}, and domain axioms in near-semirings by Struth and Desharnais \cite{a.jules08}.

Near-semirings were introduced by van Hoorn and van Rootselaar as a generalization of near-rings \cite{a.hoorn67}. In \cite{a.hoorn70}, van Hoorn generalized the concept of Jacobson radical of rings to zero-symmetric near-semirings. These radicals also generalize the radicals of near-rings by Betsch \cite{t.betsch63}. In this context, van Hoorn introduced fourteen radicals of zero-symmetric near-semiring and studied some relation between them. The properties of these radicals are further investigated in the literature (e.g. \cite{t.kvk05a,a.zulfi09}). Krishna and Chatterjee developed a radical (which is similar to the Jacobson radical of rings) for a special class of near-semirings to test the minimality of linear sequential machines in \cite{a.kvk05}.

In this paper, we study the ideals and radicals of the zero-symmetric affine near-semiring over a Brandt semigroup. First we present the necessary background material in Section 2. For the near-semiring under consideration, we obtain the right ideals in Section 3 and ascertain all radicals in Section 4. Further, we determine all its congruences and consequently obtain its ideals in Section 5.

\section{Preliminaries}

In this section, we provide a necessary background material through two subsections. One is to present the notions of near-semirings, and their ideals and radicals. In the second subsection, we recall the notion of the affine near-semiring over a Brandt semigroup. We also utilize this section to fix our notations which used throughout the work.

\subsection{A near-semiring and its radicals} In this subsection, we recall some necessary  fundamentals of near-semirings from \cite{t.kvk05a,a.hoorn70,a.hoorn67}.

\begin{definition}
An algebraic structure $(\mathcal{N}, +, \cdot)$ is said to be a \emph{near-semiring} if
\begin{enumerate}
\item $(\mathcal{N}, +)$ is a semigroup,
\item $(\mathcal{N}, \cdot)$ is a semigroup, and
\item $a \cdot (b + c) = a \cdot b + a \cdot c$, for all $a,b,c \in \mathcal{N}$.
\end{enumerate}
Furthermore, if there is an element $0 \in \mathcal{N}$ such that
\begin{enumerate}
\item [(4)] $a + 0  =  0 + a = a$ for all $a \in \mathcal{N}$, and
\item [(5)] $a \cdot 0 = 0 \cdot a = 0 $ for all $a \in \mathcal{N}$,
\end{enumerate}
then $(\mathcal{N}, +, \cdot)$ is called a \emph{zero-symmetric near-semiring}.
\end{definition}

\begin{example}
Let $(S, +)$ be a semigroup and $M(S)$ be the set of all functions on $S$. The algebraic structure $(M(S), +, \circ)$ is a near-semiring, where $+$ is point-wise addition and $\circ$ is composition of mappings, i.e., for $x \in S$ and $f,g \in M(S)$,
\[x(f + g)= x f + x g \;\;\;\; \text{and}\;\;\;\; x(f \circ g) = (x f)g.\]
We write an argument of a function on its left, e.g. $xf$ is the value of a function $f$ at an argument $x$. We always denote the composition $f \circ g$ by $fg$. The notions of homomorphism and subnear-semiring of a near-semiring can be defined in a routine way.

\end{example}

\begin{definition}
Let $\mathcal{N}$ be a zero-symmetric near-semiring. A semigroup $(S,+)$ with identity $0_S$ is said to be an \emph{$\mathcal{N}$-semigroup}  if there exists a composition \[(s, a) \mapsto sa :  S \times \mathcal{N} \longrightarrow S\] such that, for all $a , b \in \mathcal{N}$ and
$s \in S$,
\begin{enumerate}
\item $s(a+b) = s a + s b$,
\item $s(ab) =  (s a)b$, and
\item $ s 0 = 0_S$.
\end{enumerate}
\end{definition}

Note that the semigroup $(\mathcal{N}, +)$ of a near-semiring $(\mathcal{N}, +, \cdot)$ is an $\mathcal{N}$-semigroup. We denote this $\mathcal{N}$-semigroup by $\mathcal{N}^+$.

\begin{definition}
Let $S$ be an $\mathcal{N}$-semigroup. A semigroup congruence $\sim_r$ of $S$ is said to be a \emph{congruence of $\mathcal{N}$-semigroup} $S$, if for all $s, t \in S$ and $a \in \mathcal{N}$, \[s \sim_r t  \Longrightarrow sa \sim_r ta.\]
\end{definition}

\begin{definition}
An \emph{$\mathcal{N}$-morphism} of an $\mathcal{N}$-semigroup $S$ is a semigroup homomorphism $\phi$ of $S$ into an $\mathcal{N}$-semigroup $S'$
such that $$(sa)\phi = (s\phi) a$$   for all $a \in \mathcal{N}$ and $s\in S$. The kernel of an $\mathcal{N}$-morphism is called an \emph{$\mathcal{N}$-kernel} of an $\mathcal{N}$-semigroup $S$.
A subsemigroup $T$ of an $\mathcal{N}$-semigroup $S$ is said to be \emph{$\mathcal{N}$-subsemigroup} of $S$ if and only if $0_S \in T$ and $T\mathcal{N} \subseteq T$.
\end{definition}

\begin{definition}
The kernel of a homomorphism of $\mathcal{N}$ is called an \emph{ideal}  of $\mathcal{N}$. The $\mathcal{N}$-kernels of the $\mathcal{N}$-semigroup $\mathcal{N}^+$ are called \emph{right ideals} of $\mathcal{N}$.
\end{definition}

One may refer to \cite{a.hoorn70,a.hoorn67} for a few other notions viz. strong ideal, modular right ideal and $\lambda$-modular right ideal, a special congruence $r''_{\Delta}$ associated to a normal subsemigroup $\Delta$ of a semigroup $S$, and, for various $(\nu, \mu)$, the $\mathcal{N}$-semigroups of type $(\nu, \mu)$. The homomorphism corresponding to $r''_{\Delta}$ is denoted by $\lambda_\Delta$.

\begin{definition}
Let $s$ be an element of an $\mathcal{N}$-semigroup $S$. The annihilator of $s$, denoted by $A(s)$, defined by the set $\{a \in \mathcal{N}: sa = 0_S\}$. Further, for a subset $T$ of $S$, the annihilator of $T$ is
\[A(T) = \displaystyle{\bigcap_{s \in T} A(s)} = \{a \in \mathcal{N}: Ta = 0_S\}.\]
\end{definition}

\begin{theorem}[\cite{t.kvk05a}]
The annihilator $A(S)$ of an $\mathcal{N}$-semigroup $S$ is an ideal of $\mathcal{N}$.
\end{theorem}

We now recall the notions of various radicals in the following definition.

\begin{definition}[\cite{a.hoorn70}]
Let $\mathcal{N}$ be a zero-symmetric near-semiring.
\begin{enumerate}
\item For $\nu = 0,1$ with $\mu$ = 0,1,2,3 and $\nu$ = 2 with $\mu$ = 0,1 \[J_{(\nu,\mu)}(\mathcal{N}) = \displaystyle{\bigcap_{S \text{\;\;is of type} (\nu,\mu)} A(S)}.\]
\item $R_0(\mathcal{N})$ is the intersection of all maximal modular right ideals of $\mathcal{N}$.
\item $R_1(\mathcal{N})$ is the intersection of all  modular maximal right ideals of  $\mathcal{N}$.
\item $R_2(\mathcal{N})$ is the intersection of all maximal $\lambda$-modular right ideals of  $\mathcal{N}$.
\item $R_3(\mathcal{N})$ is the intersection of all  $\lambda$-modular maximal right ideals of $\mathcal{N}$.
\end{enumerate}
In any case, the empty intersection of subsets of $\mathcal{N}$ is  $\mathcal{N}$.
The relations between these radicals are given in \fref{5.f.rr.}, where $A \rightarrow B ~\mbox{means}~ A \subset B$.
\end{definition}

\begin{center}
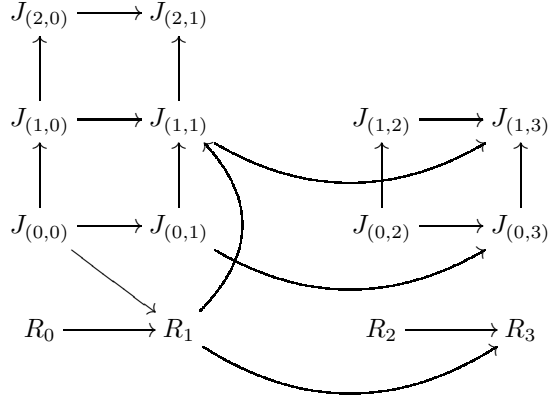
\begin{figure}[t]
\SelectTips{cm}{}
\[\xymatrix{
J_{(2,0)} \ar[r]  &   J_{(2,1)} & *\txt{} &  *\txt{} & *\txt{} & *\txt{}\\
J_{(1,0)} \ar[r] \ar[u] &  J_{(1,1)} \ar[u] \ar@/_2pc/[rrr]& *\txt{} & J_{(1,2)} \ar[r]   & J_{(1,3)}\\
J_{(0,0)} \ar[r] \ar[u] \ar[dr]&   J_{(0,1)} \ar[u]\ar@/_2pc/[rrr] & *\txt{} &  J_{(0,2)} \ar[r] \ar[u] &  J_{(0,3)} \ar[u]\\
R_0 \ar[r] &   R_1 \ar@/_2pc/[rrr] \ar@/_2pc/[uu] & *\txt{} &  R_2 \ar[r] & R_3\\ \\}\]
\caption{Relation between various radicals of a near-semiring }\label{5.f.rr.}
\end{figure}
\end{center}

\begin{remark}[\cite{t.betsch63,b.jacobson64,b.pilz}]
If $\mathcal{N}$ is a near-ring, then $J_{(0,\mu)}(\mathcal{N})$, $\mu = 0,1,2,3$ are the radical $J_0(\mathcal{N})$; $J_{(1,\mu)}(\mathcal{N})$, $\mu$ = 0,1,2,3 are the radical $J_1(\mathcal{N})$; $J_{(2,\mu)}(\mathcal{N}$), $\mu$ = 0,1, are the radical $J_2(\mathcal{N})$; and $R_{\nu}(\mathcal{N})$,  $\nu = 0,1,2,3$ are the radical $D(\mathcal{N})$ of Betsch. Further, if $\mathcal{N}$ is a ring, then all the fourteen radicals are the radical of Jacobson.
\end{remark}

\begin{definition}
A zero-symmetric near-semiring $\mathcal{N}$ is called $(\nu, \mu)$-\emph{primitive} if $\mathcal{N}$ has an $\mathcal{N}$-semigroup $S$ of type $(\nu, \mu)$ with $A(S) = \{0\}$.
\end{definition}

\subsection{An affine near-semiring over a Brandt semigroup} In this subsection, we present the necessary fundamentals of affine near-semirings over Brandt semigroups. For more detail one may refer to \cite{t.jk14,a.jk13}.

Let $(S, +)$ be a semigroup. An element $f \in M(S)$ is said to be an \emph{affine map} if $f = g + h$, for some endomorphism $g$ and a constant map $h$ on $S$. The set of all affine mappings over $S$, denoted by ${\rm Aff}(S)$, need not be a subnear-semiring of $M(S)$. The \emph{affine near-semiring}, denoted by $A^+(S)$, is the subnear-semiring generated by $\text{Aff}(S)$ in $M(S)$. Indeed, the subsemigroup of $(M(S), +)$ generated by $\text{Aff}(S)$ equals $(A^+(S), +)$ (cf. \cite[Corollary 1]{a.kvk05}). If $(S, +)$ is commutative, then $\text{Aff}(S)$ is a subnear-semiring of $M(S)$ so that ${\rm Aff}(S) = A^+(S)$.

\begin{definition}\label{d.bs}
For any integer $n \geq 1$, let $[n] = \{1,2,\ldots,n\}$. The semigroup
$(B_n, +)$, where $B_n = ([n]\times[n])\cup \{\vartheta\}$ and the
operation $+$ is given by
\[ (i,j) + (k,l) =
                \left\{\begin{array}{cl}
                (i,l) & \text {if $j = k$;}  \\
                \vartheta     & \text {if $j \neq k $}
                  \end{array}\right.  \]
and, for all $\alpha \in B_n$, $\alpha + \vartheta = \vartheta + \alpha = \vartheta$, is known as \emph{Brandt semigroup}. Note that $\vartheta$ is the (two
sided) zero element in $B_n$.
\end{definition}

Let $\vartheta$ be the zero element of the semigroup $(S, +)$. For  $f \in M(S)$, the \emph{support of $f$}, denoted by ${\rm supp}(f)$, is defined by the set
\[ {\rm supp}(f) = \{\alpha \in S \;|\; \alpha f \neq \vartheta\}.\]
A function $f \in M(S)$ is said to be of \emph{k-support} if the
cardinality of ${\rm supp}(f)$ is $k$, i.e. $|{\rm supp}(f)| = k$. If $k = |S|$ (or $k = 1$), then $f$ is said to be of \emph{full support} (or
\emph{singleton support}, respectively). For $X \subseteq M(S)$, we
write $X_k$ to denote the set of all mappings of $k$-support in $X$, i.e.
$$ X_k = \{ f \in X \mid f \; \text{is of $k$-support}\;  \}.$$

For ease of reference,  we continue to use the following notations for the elements of $M(B_n)$, as given in \cite{a.jk13}.

\begin{notation}\

\begin{enumerate}
\item For $c \in B_n$, the constant map that sends all the elements of $B_n$ to $c$ is denoted by $\xi_c$. The set of all constant maps over $B_n$ is denoted by $\mathcal{C}_{B_n}$.
\item For $k, l, p, q \in [n]$, the singleton support map that send $(k, l)$ to $(p, q)$ is denoted by $\sis{(k, l)}{(p, q)}$.
\item For $p, q \in [n]$, the $n$-support map which sends $(i, p)$ (where $1 \le i \le n$) to $(i\sigma, q)$ using a permutation $\sigma \in S_n$ is denoted by $(p, q; \sigma)$.  We denote the identity permutation on $[n]$ by $id$.
\end{enumerate}
\end{notation}

Note that $A^+(B_1) = \{(1, 1; id)\} \cup \mathcal{C}_{B_1}$. For $n \ge 2$, the elements of $A^+(B_n)$ are given by the following theorem.

\begin{theorem}[\cite{a.jk13}]\label{t.class.a+bn}
For $n \geq 2$, $A^+(B_n)$ precisely contains $(n! + 1)n^2 + n^4 + 1$ elements with the following breakup.
\begin{enumerate}
\item All the $n^2 + 1$ constant maps.
\item All the $n^4$ singleton support maps.
\item The remaining $(n!)n^2$ elements are the $n$-support maps of the form $(p, q; \sigma)$, where $p, q \in [n]$ and $\sigma \in S_n$.
\end{enumerate}
\end{theorem}

We are ready to investigate the radicals and ideals of $A^+(B_n)$ -- the affine near-semiring over a Brandt semigroup. Since the radicals are defined in the context of zero-symmetric near-semirings, we extend the semigroup reduct $(A^+(B_n), +)$ to monoid by adjoining $0$ and make the resultant near-semiring zero-symmetric. In what follows, the zero-symmetric affine near-semiring $A^+(B_n) \cup \{0\}$ is denoted by $\mathcal{N}$, i.e.
\begin{enumerate}
\item $(\mathcal{N}, +)$ is a monoid with identity element $0$,
\item $(\mathcal{N}, \circ)$ is a semigroup,
\item $0f = f0 = 0$, for all $f \in \mathcal{N}$, and
\item $f(g + h) = fg + fh$, for all $f, g, h \in \mathcal{N}$.
\end{enumerate}
In this work, a nontrivial congruence of an algebraic structure is meant to be a congruence which is neither equality nor universal relation.

\section{Right ideals}

In this section, we obtain all the right ideals of the affine near-semiring $\mathcal{N}$ by ascertaining the concerning congruences of $\mathcal{N}$-semigroups.  We begin with the following lemma.

\begin{lemma}\label{l.cong-sg}
Let $\sim$ be a nontrivial congruence over the semigroup $(\mathcal{N}, +)$ and $f \in  A^+(B_n)_{n^2 + 1}$. If $f \sim \xi_{\vt}$, then $\sim \; = (A^+(B_n) \times A^+(B_n)) \cup \{(0, 0)\}$.
\end{lemma}

\begin{proof}
First note that $(A^+(B_n) \times A^+(B_n)) \cup \{(0, 0)\}$ is a congruence relation of the semigroup $(\mathcal{N}, +)$. Let $f = \xi_{(p_0, q_0)}$ and $\xi_{(p, q)}$ be an arbitrary full support map. Since \[\xi_{(p, q)} = \xi_{(p, p_0)} + \xi_{(p_0, q_0)} + \xi_{(q_0, q)} \sim \xi_{(p, p_0)} + \xi_{\vt} + \xi_{(q_0, q)} = \xi_{\vt},\] we have $\xi_{(p, q)} \sim \xi_{\vt}$ for all $p, q \in [n]$. Further, given an arbitrary $n$-support map $(k, l; \sigma)$, since $\xi_{(p, l)} \sim \xi_{\vt}$, we have \[(k, l; \sigma) = (k, p; \sigma) + \xi_{(p, l)} \sim (k, p; \sigma) + \xi_{\vt} = \xi_{\vt}.\] Thus, all $n$-support maps are related to $\xi_{\vt}$ under $\sim$.  Similarly, given an arbitrary $\sis{(k, l)}{(p, q)} \in A^+(B_n)_1$, since $\xi_{(p, q)} \sim \xi_{\vt}$,   for $\sigma \in S_n$ such that $k\sigma = q$, we have \[\sis{(k, l)}{(p, q)} = \xi_{(p, q)} + (l, q; \sigma) \sim \xi_{\vt} + (l, q; \sigma) = \xi_\vt.\] Hence, all elements of $A^+(B_n)$ are related to each other under $\sim$.
\end{proof}

Now, using \lref{l.cong-sg}, we determine the right ideals of $\mathcal{N}$ in the following theorem.

\begin{theorem}\label{5.t.rideal}
$\mathcal{N}$ and $\{0\}$ are only the right ideals of $\mathcal{N}$.
\end{theorem}

\begin{proof}
Let $I \ne \{0\}$ be a right ideal of $\mathcal{N}$ so that $I = \ker \varphi$, where $\varphi: \mathcal{N}^+ \longrightarrow S$ is an $\mathcal{N}$-morphism. Note that $I = [0]_{ \sim_r}$, where $\sim_r$ is the congruence over the $\mathcal{N}$-semigroup $\mathcal{N}^+$ defined by $a \sim_r b$ if and only if $a\varphi = b\varphi$, i.e.  the relation $\sim_r$ on $\mathcal{N}$ is compatible with respect to $+$ and if $a \sim_r b$ then $ac \sim_r bc$ for all $c \in \mathcal{N}$.

Let $f$ be a nonzero element of $\mathcal{N}$ such that $f \sim_r 0$. First note that \[\xi_\vt = f\xi_{\vt} \sim_r 0\xi_{\vt} = 0.\] Further, for any full support map $\xi_{(p, q)}$, we have \[\xi_{(p, q)} = f\xi_{(p, q)} \sim_r 0\xi_{(p, q)} = 0\]  so that, by transitivity, $\xi_{(p, q)}  \sim_r \xi_{\vt}$. Hence, by Lemma \ref{l.cong-sg},  $\sim_r = \mathcal{N} \times \mathcal{N}$ so that $I = \mathcal{N}$.
\end{proof}

\begin{remark}\label{5.r.id0.max.}
The ideal $\{0\}$ is the maximal right ideal of $\mathcal{N}$.
\end{remark}

\section{Radicals}

In order to obtain the radicals of the affine near-semiring $\mathcal{N}$, in this section, we first identify an $\mathcal{N}$-semigroup which satisfies the criteria of all types of $\mathcal{N}$-semigroups by van Hoorn. Using the $\mathcal{N}$-semigroup, we ascertain the radicals of $\mathcal{N}$.  Further, we observe that the near-semiring $\mathcal{N}$ is $(\nu, \mu)$-primitive (cf. \tref{5.t.types}).

Consider the subsemigroup $\mathcal{C} = \mathcal{C}_{B_n} \cup \{0\}$ of $(\mathcal{N}, +)$ and observe that  $\mathcal{C}$ is an $\mathcal{N}$-semigroup with respect to the multiplication in $\mathcal{N}$. The following properties of the $\mathcal{N}$-semigroup $\mathcal{C}$ are useful.

\begin{lemma}\label{l.prop.C} $\;$
\begin{enumerate}
\item Every nonzero element of $\mathcal{C}$ is a generator. Moreover, the $\mathcal{N}$-semigroup $\mathcal{C}$ is strongly monogenic and $A(g) = \{0\}$ for all $g \in \mathcal{C} \setminus \{0\}$.
\item The subsemigroup $\{0\}$ is the maximal $\mathcal{N}$-subsemigroup of $\mathcal{C}$.
\item The $\mathcal{N}$-semigroup $\mathcal{C}$ is irreducible.
\end{enumerate}
\end{lemma}

\begin{proof}$\;$
\begin{enumerate}
\item Let $g \in \mathcal{C}_{B_n}$. Note that $g\mathcal{N} \subseteq \mathcal{C}$ because the product of a constant map with any map is a constant map. Conversely, for $f \in \mathcal{C}$, since $gf = f$, we have $g\mathcal{N} = \mathcal{C}$ for all $g \in \mathcal{C} \setminus \{0\}$. Further, since $0\mathcal{N} = \{0\}$ and $\mathcal{C}\mathcal{N}  = \mathcal{C}\ne \{0\}$. Hence, $\mathcal{C}$ is strongly monogenic.
\item We show that the semigroups $\mathcal{C}$ and $\{0\}$ are the only $\mathcal{N}$-subsemigroups of $\mathcal{C}$. Let $T$ be an $\mathcal{N}$-subsemigroup of $\mathcal{C}$ such that $\{0\} \ne T \subsetneq \mathcal{C}$. Then there exist $f (\ne 0) \in T$ and $g \in \mathcal{C} \setminus T$. Since $fg = g \not\in T$, we  have $T\mathcal{N} \nsubseteq T$; a contradiction to $T$ is an $\mathcal{N}$-subsemigroup. Hence, the result.
\item By \lref{l.prop.C}(1), the $\mathcal{N}$-semigroup $\mathcal{C}$ is monogenic with any nonzero element $g$ as generator such that $A(g) = \{0\}$; thus, $A(g)$ is maximal right ideal in $\mathcal{N}$ (cf. \rref{5.r.id0.max.}). Hence, by \cite[Theorem 8]{a.hoorn70}, $\mathcal{C}$ is irreducible.
\end{enumerate}
\end{proof}

\begin{remark}\label{r.typ 2-1}
Since a strongly monogenic $\mathcal{N}$-semigroup is monogenic we have, for $\mu = 0, 1, 2, 3$, an $\mathcal{N}$-semigroup of type $(1, \mu)$ is of type $(0, \mu)$.
\end{remark}

\begin{theorem}\label{5.t.types}
For $\nu = 0, 1$ with $\mu = 0, 1, 2, 3$ and $\nu = 2$ with $\mu = 0, 1$, we have the following.
\begin{enumerate}
 \item The $\mathcal{N}$-semigroup $\mathcal{C}$ is of type $(\nu, \mu)$ with $A(\mathcal{C}) = 0$.
 \item The near-semiring $\mathcal{N}$ is $(\nu, \mu)$-primitive for all $\nu$ and $\mu$.
 \item $J_{(\nu, \mu)}(\mathcal{N}) = \{0\}$ for all $\nu$ and $\mu$.
\end{enumerate}
\end{theorem}

\begin{proof}  In view of \rref{r.typ 2-1}, we prove (1) in the following cases.
\begin{description}
\item[Type $(1, \mu)$] Note that, by  \lref{l.prop.C}(1), the  $\mathcal{N}$-semigroup $\mathcal{C}$ is strongly monogenic.
    \begin{enumerate}
    \item[(i)] By \lref{l.prop.C}(3), we have $\mathcal{C}$ is irreducible. Hence, $\mathcal{C}$ is of type $(1,0)$.

    \item[(ii)] By \lref{l.prop.C}(1) and \rref{5.r.id0.max.}, for any generator $g$, $A(g)$ is a maximal right ideal. Hence, $\mathcal{C}$ is of type $(1,1)$.

    \item[(iii)] Note that the ideal $\{0\}$ is strong right ideal so that for any generator $g$, $A(g)$ is a strong maximal right ideal (see ii above). Further, note that  $A(g)$ is a  maximal strong right ideal (cf. \rref{5.r.id0.max.}). Hence, $\mathcal{C}$ is of type $(1, 2)$ and $(1, 3)$.
\end{enumerate}

\item[Type $(2, \mu)$] Since $\mathcal{C}$ is monogenic and, for any generator $g$ of $\mathcal{C}$, $A(g)$ is a maximal $\mathcal{N}$-subsemigroup of $\mathcal{C}$ (cf. \lref{l.prop.C}(1) and \lref{l.prop.C}(2)).  Thus, $\mathcal{C}$ is of type $(2, 1)$. By \cite[Theorem 9]{a.hoorn70}, every $\mathcal{N}$-semigroup of type $(2,1)$ will be of type $(2,0)$. Hence, $\mathcal{C}$ is of type $(2,0)$.
\end{description}
Proofs for (2) and (3) follow from (1).
\end{proof}

\begin{theorem}\label{5.t.rad-2}
 For $\nu = 0,1$, we have $R_{\nu}(\mathcal{N}) = \{0\}$.
\end{theorem}

\begin{proof}
In view of \fref{5.f.rr.}, we prove that result by showing that the right ideal $\{0\}$ is a modular maximal right ideal. By \lref{l.prop.C}(1), the $\mathcal{N}$-semigroup  $\mathcal{C}$ is monogenic and has a generator $g$ such that $A(g) = \{0\}$. Hence, the right ideal $\{0\}$ is modular (cf. \cite[Theorem 7]{a.hoorn70}). Further, since $\{0\}$ is a maximal right ideal (cf. \rref{5.r.id0.max.}), we have $\{0\}$ is a modular maximal right ideal.
\end{proof}

\begin{theorem}\label{5.t.rad-3}
For $\nu = 2, 3$, we have $R_{\nu}(\mathcal{N}) = \mathcal{N}$.
\end{theorem}

\begin{proof}
In view of \fref{5.f.rr.} and \tref{5.t.rideal}, we prove that the homomorphism $\lambda_{\{0\}}$  is not modular. Note that the congruence relation $r_{\{0\}}''$ is the equality relation on $(\mathcal{N}, +)$, where $r_{\{0\}}''$ is the transitive closure of the two sided stable reflexive and symmetric relation relation $r_{\{0\}}$ associated with a normal subsemigroup $\{0\}$ of the semigroup $(\mathcal{N}, +)$. Consequently, the semigroup homomorphism $\lambda_{\{0\}}$ is an identity map on $(\mathcal{N}, +)$. If the morphism $\lambda_{\{0\}}$ is modular, then there is an element $u \in \mathcal{N}$ such that $x = ux$ for all $x \in \mathcal{N}$, but there is no left identity element in $\mathcal{N}$. Consequently, $\lambda_{\{0\}}$ is not modular. Thus, there is no maximal $\lambda$-modular right ideal. Hence, for $\nu = 2, 3$, we have $R_{\nu}(\mathcal{N}) = \mathcal{N}$.
\end{proof}

\section{Ideals}

In this section, we prove that there is only one nontrivial congruence relation on $\mathcal{N}$ (cf. \tref{5.t.congs}). Consequently, all the ideals of $\mathcal{N}$ are determined.

\begin{theorem}\label{5.t.congs}
The near-semiring $\mathcal{N}$ has precisely the following congruences.
\begin{enumerate}
\item Equality relation
\item $\mathcal{N} \times \mathcal{N}$
\item $(A^+(B_n) \times A^+(B_n)) \cup \{(0, 0)\}$
\end{enumerate}
Hence, $\mathcal{N}$ and $\{0\}$ are the only ideals of the near-semiring $\mathcal{N}$.
\end{theorem}

\begin{proof}
In the sequel, we prove the theorem through the following claims.

\textit{Claim 1}: Let $\sim$ be a nontrivial congruence over the near-semiring $\mathcal{N}$ and $f \in \mathcal{N} \setminus \{0, \xi_{\vt}\}$. If $f \sim \xi_{\vt}$, then $\sim \; = (A^+(B_n) \times A^+(B_n)) \cup \{(0, 0)\}$.

\textit{Proof}: First note that $(A^+(B_n) \times A^+(B_n)) \cup \{(0, 0)\}$ is a congruence relation of the near-semiring $\mathcal{N}$. If $f \in A^+(B_n)_{n^2 + 1}$, since $\sim$ is a congruence of the semigroup $(\mathcal{N}, +)$, by Lemma \ref{l.cong-sg}, we have the result. Otherwise, we reduce the problem to Lemma \ref{l.cong-sg} in the following cases.

\begin{description}
\item[{\rm\em Case 1.1}] $f$ is of singleton support. Let $f$ = $\sis {(k, l)}{(p, q)}$. Since $\sis {(k, l)}{(p, q)} \sim \xi_{\vt}$ we have \[\xi_{(k, l)}\;\sis {(k, l)}{(p, q)} \sim \xi_{(k, l)}\xi_{\vt}\] so that $\xi_{(p, q)} \sim \xi_{\vt}$.

\item[{\rm\em Case 1.2}] $f$ is of  $n$-support. Let $f = (p, q; \sigma)$. Since $(p, q; \sigma) \sim \xi_{\vt}$ we have \[\xi_{(k, p)}(p, q; \sigma) \sim \xi_{(k, p)}\xi_{\vt}\] so that $\xi_{(k\sigma, q)} \sim \xi_{\vt}$.
\end{description}

\textit{Claim 2}: If two nonzero elements are in one class under a nontrivial congruence over $\mathcal{N}$, then the congruence is $(A^+(B_n) \times A^+(B_n)) \cup \{(0, 0)\}$.

\textit{Proof}:
Let $f, g \in \mathcal{N}\setminus \{0\}$ such that $f \sim g$ under a congruence $\sim$ over $\mathcal{N}$. If $f$ or $g$ is equal to $\xi_{\vt}$, then by \textit{Claim 1}, we have the result. Otherwise, we consider the following six cases classified by the supports of $f$ and $g$. In each case, we show that there is an  element $h \in A^+(B_n)\setminus \{\xi_{\vt}\}$  such that $h \sim \xi_{\vt}$ so that the result follows from \textit{Claim 1}.
\begin{description}
  \item [{\rm\em Case 2.1}] $f, g \in A^+(B_n)_1$.
  Let $f$ = $\sis {(i, j)}{(k, l)}$ and $g$ = $\sis {(s, t)}{(u, v)}$. If $(i, j) \ne (s, t)$, we have \[\xi_{\vt} =\; \sis{(i, j)}{(k, l)} + \sis{(s, t)}{(v, v)} \sim \; \sis {(s, t)}{(u, v)} +\; \sis{(s, t)}{(v, v)} =\; \sis{(s, t)}{(u, v)}.\]

   Otherwise, $(i, j) = (s, t)$ so that $(k, l) \ne (u, v)$. Now, if $k \ne u$, then we have \[\sis {(i, j)}{(k, l)} =\; \sis{(i, j)}{(k, k)} +\; \sis{(i, j)}{(k, l)} \sim \;\sis{(i, j)}{(k, k)} +\; \sis {(i, j)}{(u, v)} = \xi_{\vt}.\] Similarly, if $l \ne v$, we have \[\xi_\vt =\; \sis{(i, j)}{(k, l)} +\; \sis{(i, j)}{(v, v)} \sim\; \sis {(i, j)}{(u, v)} +\; \sis{(i, j)}{(v, v)} = \; \sis{(i, j)}{(u, v)}.\]

   \item [{\rm\em Case 2.2}] $f, g \in A^+(B_n)_{n^2+1}$.
   Let $f = \xi_{(k, l)}$ and  $g = \xi_{(u, v)}$. By considering full support maps whose images are the same as in various subcases of {\em Case 1}, we can show that there is an element in $A^+(B_n)\setminus \{\xi_{\vt}\}$ that is related to $\xi_\vt$ under $\sim$.

   \item [{\rm\em Case 2.3}] $f, g \in A^+(B_n)_n$.
   Let $f = (i, j; \sigma)$ and $g = (k, l; \rho)$. If $l \ne j$, then \[(i, j; \sigma) = (i, j; \sigma) + \xi_{(j, j)} \sim (k, l; \rho) + \xi_{(j, j)} = \xi_{\vt}.\] Otherwise, we have $(i, j; \sigma) \sim (k, j; \rho)$. Now, if $i \ne k$, then \[ \xi_\vt = (k, k; id)(i, j; \sigma) \sim (k, k; id)(k, j; \rho) = (k, j; \rho).\] In case $i = k$, we have $\sigma \ne \rho$. Thus, there exists $t \in [n]$ such that $t \sigma \ne t \rho$. Now, $(i, j; \sigma) \sim (i, j; \rho)$ implies $\xi_{(k, i)}(i, j; \sigma) \sim \xi_{(k, i)}(i, j; \rho)$, i.e. $\xi_{(k\sigma, j)} \sim \xi_{(k\rho, j)}$. Consequently, \[\xi_{(k\sigma, j)} = \xi_{(k\sigma, k\sigma)} + \xi_{(k\sigma, j)} \sim \xi_{(k\sigma, k\sigma)} + \xi_{(k\rho, j)} = \xi_{\vt}.\]

   \item [{\rm\em Case 2.4}] $f \in A^+(B_n)_1, g \in A^+(B_n)_{n^2+1}$.
   Let $f =\; \sis {(k, l)}{(p, q)}$ and $g = \xi_{(i, j)}$. Now, for $(s, t) \ne (k, l)$, we have \[\xi_\vt = \xi_{(s, t)}f \sim \xi_{(s, t)}g = \xi_{(i, j)}.\]

   \item [{\rm\em Case 2.5}] $f \in A^+(B_n)_{n^2+1}, g \in A^+(B_n)_n$.
   Let $f = \xi_{(p, q)}$ and $g = (i, j; \sigma)$. Now, for $l \ne i$, we have \[\sis {(k, l)}{(p, q)} =\; \sis {(k, l)}{(p, p)} +  f \sim\; \sis {(k, l)}{(p, p)} + g = \xi_{\vt}.\]

   \item [{\rm\em Case 2.6}] $f \in A^+(B_n)_1, g \in A^+(B_n)_n$.
   Let $f =\; \sis {(k, l)}{(p, q)}$ and $g = (i, j; \sigma)$. Now, for $l \ne i$, we have \[\xi_\vt = \xi_{(i, i)}f \sim \xi_{(i, i)}g = \xi_{(i\sigma, j)}.\]

\end{description}
\end{proof}

\end{document}